\documentclass[11pt,a4paper,oneside,english]{amsart}
\usepackage[T1]{fontenc}
\usepackage[latin9]{inputenc}
\usepackage{mathtools}
\usepackage{amstext}
\usepackage{amsthm}
\usepackage{amssymb}

\makeatletter

\pdfpageheight\paperheight
\pdfpagewidth\paperwidth

\numberwithin{equation}{section}
\numberwithin{figure}{section}
\theoremstyle{plain}
\newtheorem{thm}{\protect\theoremname}
\theoremstyle{definition}
\newtheorem{defn}[thm]{\protect\definitionname}
\newenvironment{lyxlist}[1]
	{\begin{list}{}
		{\settowidth{\labelwidth}{#1}
		 \setlength{\leftmargin}{\labelwidth}
		 \addtolength{\leftmargin}{\labelsep}
		 }}
	{\end{list}}
\theoremstyle{plain}
\newtheorem{lem}[thm]{\protect\lemmaname}
\theoremstyle{remark}
\newtheorem{rem}[thm]{\protect\remarkname}
\theoremstyle{definition}
\newtheorem{example}[thm]{\protect\examplename}
\theoremstyle{plain}
\newtheorem{cor}[thm]{\protect\corollaryname}

\makeatother

\usepackage{babel}
\providecommand{\corollaryname}{Corollary}
\providecommand{\definitionname}{Definition}
\providecommand{\examplename}{Example}
\providecommand{\lemmaname}{Lemma}
\providecommand{\remarkname}{Remark}
\providecommand{\theoremname}{Theorem}

\begin{document}
\title{Simplicial and Conical Decomposition of Positively Spanning Sets}
\date{November $1^{\mathrm{st}}$ 2019}
\author{Daniel Schoch}
\address{Nottingham University Malaysia Campus, School of Economics}
\begin{abstract}
We investigate the decomposition of a set $X$, which positively spans
the Euclidean space $\mathbb{R}^{d}$ into a set of minimal positive
bases, we call simplices, and into maximal sets positively spanning
pointed cones, i.e. cones with exactly one apex. For any set $X$,
let $\mathcal{S}\left(X\right)$ denote the set of simplex subsets
of $X$, and let $\ell\left(X\right)$ denotes the linear hull of
$X$. The set $X$ is said to fulfill the factorisation condition
if and only if for each subset $Y\subset X$ and each simplex $S\in\mathcal{S}\left(X\right)$,
$\ell\left(Y\right)\cap\ell\left(S\right)=\ell\left(Y\cap S\right)$.

We demonstrate that $X$ is a positive basis if and only if it is
the union of most $d$ simplices, and $X$ satisfies the factorization
condition. In this case, $X$ contains a linear basis $B$ such that
each simplex in $\mathcal{S}\left(X\right)$ has with $B$, all but
one exactly one element in common. We show that for sets positively
spanning $\mathbb{R}^{d}$, the set of subbases of $X$ forms a boolean
lattice, which can be embedded into the set $2^{\mathcal{S}\left(X\right)}$,
with isomorphy for positive bases.

Our second main result depending on the former is as follows. A finite
set $X\subset\mathbb{R}^{d}\setminus\left\{ 0\right\} $ can be written
as the union of at most $2^{d}$ maximal sets spanning pointed cones,
which, if $X$ is a positive basis, are tantamount to frames of the
cones. The inequality holds sharply if and only if $X$ is a cross,
that is, a union of 1-simplices derived from a linear basis of $\mathbb{R}^{d}$.
We also show that there can be at the most $2^{d}$ maximal subsets
of $X$ spanning pointed cones, when intersections of two of them
do not span a set of full dimension.
\end{abstract}

\thanks{I owe thanks to C.S. for proofreading the manuscript.}
\maketitle

\section{Introduction}

Since Chandler Davis's seminal paper \cite{DAV54}, analysis of the
structure of sets of vectors positively generating solid cones (i.e.
linear spaces) has gained attention. The attempts to understand the
structure of such sets has been focused on positive bases, which are
minimal positive generating sets. These approaches involved the decomposition
into disjointed subsets \cite{REA65}, analysis through the Gale transformation
and related techniques \cite{SHE71,MAR84}, and generation through
certain kind of matrices \cite{DAV54}. The latter results include
the only characterizations of positive bases; however technical, and
beyond the obvious equvalence of the minimal generating sets being
positively independent.

Our approach is novel as far as it focuses on positive linear relations,
equations yielding zero, instead of linear relations. We identified
a simple and intuitive characterization of positive bases (theorem
\ref{thm:Inter}). Building on this, we show that each positive spanning
set (not only bases) can be decomposed into not greater than $2^{d}$
sets that positively spanns pointed cones.

Notation. All spaces $\mathbb{R}^{d}$ are Euclidean. By $\ell\left(X\right)$
and $\wp\left(X\right)$, we denote the linear and positive span of
$X$, respectively, i.e. the set of linear combinations of the following
form:
\[
a_{1}x_{1}+\cdots+a_{k}x_{k},\;x_{1},\ldots,x_{k}\in X
\]
with real and non-negative coefficients $a_{1},\ldots,a_{n}$, respectively.
By convention, the empty set spans, positively and linearly, the null
space $\left\{ 0\right\} $. For general properties of $\wp\left(.\right)$
see \cite{REG15}. A positive spanning set (PSS) is any set $X$ that
positively spans a linear space \cite{MAR84}. A positive basis is
a PSS with no proper subset that spans the same space. $\left|X\right|$
denotes the cardinality of $X$. For $A\subset\mathbb{R}^{d}$, $\mathrm{rint}\left(A\right)$
is the interior relative to the affine space spanned by $A$.
\begin{defn}
A set $X\subset\mathbb{R}^{d}$ is called linearly / positively /
negatively dependent, if and only if for some $x\in X$ we have the
relation $x\in\ell\left(X\setminus\left\{ x\right\} \right)$ / $x\in\wp\left(X\setminus\left\{ x\right\} \right)$
/ $-x\in\wp\left(X\setminus\left\{ x\right\} \right)$, respectively.
It is called linearly / positively / negatively independent, if it
is not dependent.
\end{defn}

The upper definition of linear dependence is obviously equivalent
to the standard definition. Linear independence implies positive and
negative independent, but the converse does not hold save in dimension
2 (lemma \ref{lem:LInd2D} and example \ref{exa:LIndnD}). Negative
independence of a set means that the positive cone generated by this
set is pointed with $0$ as the only apex; another equivalent term
is strictly one-sided \cite[p. 112]{BOL97} (see lemma \ref{lem:NegInd}).
\begin{defn}
A set $S\subset\mathbb{R}^{d}$ is called a simplex basis (or simplex,
for short) if and only if $0\in\wp\left(S\right)$, and no proper
subset has this property. For a set $X$, $\mathcal{S}\left(X\right)$
denote the set of all simplices contained in $X$.

In the literature, the term minimal basis refers to the simplex basis
\cite{DAV54}; However, this notation may lead yo confusion, because
the positive basis itself is a minimal positively spanning set. The
term is justified as a simplex basis and can be understood as the
vertices of a geometrical simplex with its origin in its relative
interior. For example, a 1-simplex consists of two opposite vectors
$x,-\alpha x$, $x\neq0$, $\alpha>0$.

The opposite extreme of a basis is the cross. This is generated from
a linear basis $B$ of the space by adding a negative multiple, forming
a 1-simplex, to each element of the basis. The convex hull of these
simplices all intersect at the origin. While the cardinality of a
simplex spanning $\mathbb{R}^{d}$ is always $d+1$, the cardinality
of a cross spanning $\mathbb{R}^{d}$ is $2d$ \cite{DAV54}. It is
well know that these are unique instances that yield sharp boundaries
for the cardinality of positive bases \cite{DAV54,HAR15,REG15}.

It is easy to see that each positively spanning set is the union of
its simplices (theorem \ref{thm:Gen}). For a positive basis, these
simplices overlap in common linear subspaces, as the following theorem
shows.
\end{defn}

\begin{thm}
\label{thm:Inter}The following statements are equivalent for a set
$X$ positively spanning $\mathbb{R}^{d}$:
\end{thm}

\begin{lyxlist}{00.00.0000}
\item [{(i)}] $X$ is positively independent.
\item [{(ii)}] For all subsets $Y\subset X$ and all simplices $S\in\mathcal{S}\left(X\right)$
\[
\ell\left(Y\right)\cap\ell\left(S\right)=\ell\left(Y\cap S\right).
\]
\item [{(iii)}] For all positively spanning subsets $Y\subset X$ and all
simplices $S\in\mathcal{S}\left(X\right)$ 
\[
\ell\left(Y\right)\cap\ell\left(S\right)=\ell\left(Y\cap S\right).
\]
\item [{(iv)}] We can write $X=B\cup\left\{ x_{1},\ldots,x_{n}\right\} $,
$1\leq n\leq d$, with a basis $B$ of $\ell\left(X\right)$, such
that each $x_{i}\in-\mathrm{rint}\wp\left(A_{i}\right)$ for some
$A_{i}\subset B$, and $A_{i}\nsubseteq A_{j}$ for $i\neq j$. Moreover,
each simplex $S\in\mathcal{S}\left(X\right)$ has all but one element
in common with $B$.
\end{lyxlist}
Following Reay \cite{REA65}, the Bonnice-Klee theorem can be derived
from a corollary to condition (iv) (corollary \ref{cor:Reay}). Our
first main result is the previous theorem along with the following:
We show that for any positively spanning set $X$, the set of subsets
positively generating linear subspaces form a boolean lattice isomorphic
to a sublattice of $2^{\mathcal{S}\left(X\right)}$, with isomorphy
if $X$ is a positive basis (theorem \ref{thm:Lat} and corollary
\ref{cor:Iso}).

The second main result emanates from the following theorem.
\begin{defn}
By $\mathcal{M}\left(X\right)$, we denote the set of all maximal
negatively independent subsets of $X$.
\end{defn}

\begin{thm}
\label{thm:Main}Let $X$ be a positive basis of $\mathbb{R}^{d}$.
Then $X$ can be written as the union of $n$ simplices such that
the following inequalities hold
\begin{eqnarray*}
1\leq & n & \leq d,\\
d+1\leq & \left|X\right| & \leq2d,\\
d+1\leq & \left|\mathcal{M}\left(X\right)\right| & \leq2^{d},
\end{eqnarray*}
with equality in each of the upper inequalities (and therefore in
all) if and only if $X$ is a cross, and equalityin one or all of
the lower equations if and only if $X$ is a simplex.
\end{thm}

There are two extensions for the last inequality. We show that each
positively spanning set can be decomposed into not more than $2^{d}$
sets that positively span pointed cones (theorem \ref{thm:MainExt}).
Our last theorem (and second main result) extends the last inequality
to the case when $X$ is not necessarily a positive basis. Another
way to state this is that there can be no more than $2^{d}$ maximally
pointed $d$-cones in $\mathbb{R}^{d}$ (maximally pointed in the
sense that no element of the frame of another cone can be added to
it and keep it pointed), such that two cones overlap only at the boundary.
\begin{thm}
\label{thm:Max}Let $X$ positively span $\mathbb{R}^{d}$, and $\mathcal{A}\subset\mathcal{M}\left(X\right)$
such that for all $A,B\in\mathcal{A}$, $A\neq B$, $\wp\left(A\cap B\right)$
does not have full dimension. Then $\left|\mathcal{A}\right|\leq2^{d}$,
with equality if and only if $X$ is a union of at least $d$ 1-simplices,
some $d$ of which form a cross.
\end{thm}

\section{Preliminaries}

More notation. For any set $A$ we write $A\left[x\rightarrow y\right]$
for $\left(A\setminus\left\{ x\right\} \right)\cup\left\{ y\right\} $.
Let $\mathrm{conv}\left(X\right)$ denote the convex hull of $X$.
\begin{lem}
\label{lem:NegInd}For a set $X\subset\mathbb{R}^{d}$ the following
is equivalent.
\end{lem}

\begin{description}
\item [{(i)}] $X$ is negatively independent,
\item [{(ii)}] $\mathcal{S}\left(X\right)=\emptyset$,
\item [{(iii)}] There is a $z$ with $z\cdot x>0$ for all $x\in X$.
\end{description}
\begin{proof}
``(i)$\Rightarrow$(ii)'': Let $S\in\mathcal{S}\left(X\right)$
with $S=\left\{ x_{1},\ldots,x_{n}\right\} $. Then there are $\alpha_{i}\geq0$,
not all zero, with $\sum_{i=1}^{n}\alpha_{i}x_{i}=0$. Hence $x_{k}=-\sum_{i=1}^{n}\left(\alpha_{i}/\alpha_{k}\right)x_{k}$
for one $k$. Thus $X$ is negatively dependent.

``(ii)$\Rightarrow$(iii)'': Assume that for all $z$ there is an
$x\in X$ with $z\cdot x\leq0$. Let $C=\wp\left(X\right)$. If $C$
and $-C$ intersected only in $0$, then $C\setminus\left\{ 0\right\} $
and $-C\setminus\left\{ 0\right\} $ could be strictly separated by
a hyperplane through the origin. Let $z$ be the normal vector orthogonal
to the hyperplane lying in the half space containing $C$, then $z\cdot x>0$
for all $x\in C\setminus\left\{ 0\right\} $, a contradiction to the
assumption. Thus $C$ and $-C$ contain a common nonzero vector $x$.
There are $x_{1},\ldots,x_{m}\in X$ and $\alpha_{1},\ldots,\alpha_{m}>0$
such that 
\[
x=\sum_{i=1}^{k}\alpha_{i}x_{i},\text{ and }-x=\sum_{i=k+1}^{m}\alpha_{i}x_{i},
\]
and $\sum_{i=1}^{m}\alpha_{i}x_{i}=0$. Thus $\left\{ x_{1},\ldots,x_{m}\right\} $
contain a minimal set with this property, which forms a simplex $S\in\mathcal{S}\left(X\right)$.

``(iii)$\Rightarrow$(i)'': Assume $z\cdot x>0$ for all $x\in X$
for some $z\in\mathbb{R}^{d}\setminus\left\{ 0\right\} $. Assume
further $X$ is negatively dependent. Then there is $x_{0},\ldots,x_{n}\in X$
and $\alpha_{1},\ldots,\alpha_{n}>0$ with $x_{0}=-\sum_{i=1}^{n}\alpha_{i}x_{i}$.
But then, $z\cdot x_{0}>0$ and $z\cdot\sum_{i=1}^{n}\alpha_{i}x_{i}>0$,
a contradiction. Hence $X$ is negatively independent.
\end{proof}
The following result is trivial.
\begin{rem}
\label{rem:Bas}Let $X=\left\{ x_{1},\ldots,x_{m}\right\} $ with
$m\leq d=\mathrm{dim}\wp\left(X\right)$. Then $X$ is a linear basis
of $\ell\left(X\right)$.
\end{rem}

\begin{proof}
We choose a maximal linear independent subset $B\subset X$. Then
$\ell\left(B\right)=\ell\left(X\right)$, and $\mathrm{dim}\ell\left(B\right)\geq d$.
But this requires $\left|B\right|\geq d$, hence $B=X$ and $m=d$.
\end{proof}
Linearly independent sets are positively and negatively independent.
The converse holds only in two dimensions, as the following two results
show.
\begin{lem}
\label{lem:LInd2D}Let $X\subset\mathbb{R}^{2}\setminus\left\{ 0\right\} $
. Then $X$ is linearly independent if and only if it is both positively
and negatively independent.
\end{lem}

\begin{proof}
Only sufficiency has to be shown. Let us first assume that $X=\left\{ x_{1},\ldots,x_{n}\right\} $
is positively and negatively independent. We have to show that $X$
is linearly independent. For $n=1$ this is trivial. For $n\geq2$,
this means $x_{1}\neq\alpha x_{2}$ and $x_{1}\neq-\alpha x_{2}$
for all $\alpha>0$, as $x_{1},x_{2}\neq0$, hence $x_{1}\neq ax_{2}$
for any $a\in\mathbb{R}$, and $\left\{ x_{1},x_{2}\right\} $ are
linearly independent and span $\mathbb{R}^{2}$ . For $n\geq3$ we
therefore have $x_{3}=ax_{1}+bx_{2}$ for $a,b$ not both zero. If
$a,b\geq0$, then $X$ positively dependent. If $a,b\leq0$, then
$X$ is negatively dependent. If one coefficient is positive, and
one negative, say $a>0$ and $b<0$, then $x_{1}=\frac{1}{a}x_{3}+\frac{\left|b\right|}{a}x_{2}$,
and $X$ is positively dependent. Hence $n\leq2$ and $X$ is linearly
independent. This also covers the infinite case.
\end{proof}
We only need a counterexample in three dimensions.
\begin{example}
\label{exa:LIndnD}Let $e_{1},e_{2},e_{3}$ be the standard basis
in $\mathbb{R}^{3}$ and $z=e_{1}+e_{2}-e_{3}$. Then $X=\left\{ e_{1},e_{2},e_{3},z\right\} $
is both positively and negatively independent, but linearly dependent.
\end{example}

\begin{proof}
We observe that for $i\neq j$, $\left\{ e_{i},e_{j},z\right\} $
is a linear basis of $\mathbb{R}^{3}$. As all elements of $X$ are
linear combinations of the others,
\begin{eqnarray*}
e_{1} & = & z-e_{2}+e_{3},\\
e_{2} & = & z-e_{1}+e_{2},\\
e_{3} & = & z-e_{1}-e_{2},\\
z & = & e_{1}+e_{2}-e_{3},
\end{eqnarray*}
with both positive and negative coefficients, $X$ is both negatively
and positively independent.
\end{proof}
We show Caratheodory's theorem in a version for positive linear combinations
in cones.
\begin{lem}
\label{lem:Car}Let $X\subset\mathbb{R}^{d}$ and $x\in\wp\left(X\right)$.
Then there are $x_{1},\ldots,x_{n}\in X$ with $n\leq d$ and $\alpha_{i}>0$
such that
\[
x=\sum_{i=1}^{n}\alpha_{i}x_{i}
\]
\end{lem}

\begin{proof}
Let $x\in\wp\left(X\right)$, such that $x=\sum_{i=1}^{n}\alpha_{i}x_{i}$
for $x_{i}\in X$ and, without loss of generality, $\alpha_{i}>0$.
The set $B=\left\{ x_{1},\ldots,x_{n}\right\} $ can be chosen negatively
independent: If $x_{k}=-\sum_{i\neq k}\beta_{i}x_{i}$, letting $\beta_{k}=1$
and $I=\left\{ i\mid\beta_{i}>0\right\} $, select an $l\in J$ with
$\gamma\coloneqq\alpha_{l}/\beta_{l}\leq\alpha_{i}/\beta_{i}$ for
all $i\in I$. Then $\sum_{i\in I}\beta_{i}x_{i}=0$ and letting $\beta_{i}=0$
for $i\notin I$, we obtain
\[
x=x-\gamma\sum_{i\in I}\beta_{i}x_{i}=\sum_{i=1}^{n}\left(\alpha_{i}-\frac{\alpha_{l}}{\beta_{l}}\beta_{i}\right)x_{i}=\sum_{i\neq l}\left(\alpha_{i}-\frac{\alpha_{l}}{\beta_{l}}\beta_{i}\right)x_{i},
\]
where the coefficients are $\alpha_{i}-\left(\alpha_{l}/\beta_{l}\right)\beta_{i}\geq\alpha_{i}-\left(\alpha_{i}/\beta_{i}\right)\beta_{i}=0$.

By lemma \ref{lem:NegInd} there is a $z\neq0$ with $z\cdot x_{i}>0$.
By $H=\left\{ y\in\mathbb{R}^{d}\mid z\cdot y=z\cdot x\right\} $
we define a hyperplane such that for each $i=1,\ldots,n$ there is
a $\gamma_{i}>0$ with $y_{i}=\gamma_{i}x_{i}\in H$. Since $x\in H$,
there are now $\beta_{i}=\alpha_{i}/\gamma_{i}>0$ with $x=\sum_{i=1}^{n}\beta_{i}y_{i}$.
Since $x\in\wp\left(B\right)\cap H=\mathrm{conv}\left(\left\{ y_{1},\ldots,y_{n}\right\} \right)$,
by Caratheodory's theorem, a subset of $d$ elements of $\left\{ y_{1},\ldots,y_{n}\right\} $
suffice to have $x$ in its convex hull. Hence $n\leq d$ can be chosen.
\end{proof}

\section{Simplicial Decompositions and Gale Diagrams}
\begin{lem}
\label{lem:Sim}The following statements are equivalent.
\end{lem}

\begin{description}
\item [{(i)}] $S$ is a simplex.
\item [{(ii)}] For all $z\in S$, $B_{z}=S\setminus\left\{ z\right\} $
is linearly independent with $z\in-\mathrm{rint}\left(\wp\left(B_{z}\right)\right)$.
\item [{(iii)}] There is a $z\in S$ such that $B_{z}=S\setminus\left\{ z\right\} $
is linearly independent, and $z\in-\mathrm{rint}\left(\wp\left(B_{z}\right)\right)$.
\item [{(iv)}] $S$ is a minimal set with $\wp\left(S\right)=\ell\left(S\right)$.
\end{description}
\begin{proof}
``(i)$\Rightarrow$(ii)'': Let $S=\left\{ x_{0},\ldots,x_{m}\right\} $
be a simplex. Then, by definition, there are $\alpha_{i}>0$ with
$\sum_{i=0}^{m}\alpha_{i}x_{i}=0$. Fix any $x_{k}$ and set $B_{k}=S\setminus\left\{ x_{k}\right\} $.
Then
\[
-x_{k}=\sum_{i\neq k}\frac{\alpha_{i}}{\alpha_{k}}x_{i}\in\mathrm{rint}\left(\wp\left(B_{k}\right)\right).
\]
We show that $B_{k}$ is linearly independent by induction over $m=\left|B_{k}\right|$.
For $m=1$, which is trivial. We assume now the proposition has been
shown for $m-1$. Choose $x_{l}\in B_{k}$ and set $B_{kl}=B_{k}\setminus\left\{ x_{l}\right\} $.
By induction hypothesis, $B_{kl}$ is linearly independent. To show
that $B_{k}$ is linearly independent, it is sufficient to demonstrate
that $x_{l}\notin\ell\left(B_{kl}\right)$. Assume the contrary, that
is, $x_{l}\in\ell\left(B_{kl}\right)=\ell\left(B_{k}\right)$. With
$-x_{k}\in\wp\left(B_{k}\right)\subset\ell\left(B_{k}\right)=\ell\left(B_{kl}\right)$,
we find $S\subset\ell\left(B_{kl}\right)$. Since $0\in\wp\left(S\right)$
with $m-1=\left|B_{kl}\right|=\mathrm{dim}\ell\left(B_{kl}\right)$,
by Caratheodory's theorem for cones (lemma \ref{lem:Car}), either
$0\in\wp\left(B_{kl}\cup\left\{ x_{l}\right\} \right)$, or $0\in\wp\left(B_{kl}\cup\left\{ x_{k}\right\} \right)$.
In either case, $S$ contains a proper subset with zero in its positive
span, a contradiction to $S$ being a simplex.

``(ii)$\Rightarrow$(iii)'': Trivial.

``(iii)$\Rightarrow$(ii)'': Assume, without loss, that $x_{0}\in S=\left\{ x_{0},\ldots,x_{m}\right\} $
with $x_{0}\in-\mathrm{rint}\left(\wp\left(B\right)\right)$, and
$B=\left\{ x_{1},\ldots,x_{m}\right\} $ is linearly independent.
Thus, 
\[
-x_{0}=\sum_{i=1}^{m}\alpha_{i}x_{i}
\]
 is a unique linear combination and $\alpha_{i}>0$. Fix any $x_{k}\in B$
and set $B'=B\left[x_{k}\rightarrow x_{0}\right]$. If $B'$ is linearly
dependent, then $x_{0}\in\ell\left(B'\setminus\left\{ x_{0}\right\} \right)$,
and further $x_{0}=\sum_{i\neq k}a_{i}x_{i}$, which implies $\alpha_{k}=0$,
a contradiction. Moreover,
\[
x_{0}+\alpha_{k}x_{k}+\sum_{1\leq i\neq k}\alpha_{i}x_{i}=0.
\]
Setting $\alpha_{0}=1$ we find that
\[
-x_{k}=\sum_{0\leq i\neq k}\alpha_{i}x_{i}\in\mathrm{rint}\wp\left(B'\right),
\]
which proves the assertion.

``(ii)$\Rightarrow$(iv)'': Assume that for any $z\in S$, $B_{z}=S\setminus\left\{ z\right\} $
is linearly independent with $z\in-\mathrm{rint}\left(\wp\left(B_{z}\right)\right)$.
\end{proof}
The main proposition of the following theorem with positively spanning
sets are the sums of simplices is not new; however, can be obtained
from Blaschke's Theorem on decomposition of polytopes into simplices
\cite[Ch 15.3]{GRU03}. The additional requirement by Minkowski's
Theorem of equilibration of the vectors in a positively spanning set
(summing up to zero) that indicates no loss of generality for our
result.
\begin{thm}
\label{thm:Gen}For a finite set $X\subset\mathbb{R}^{d}$, the following
statements are equivalent.
\end{thm}

\begin{description}
\item [{(i)}] $X$ is a positively spanning set ($\wp\left(X\right)=\ell\left(X\right)$).
\item [{(ii)}] $\wp\left(X\right)=-\wp\left(X\right)$.
\item [{(iii)}] $X=\bigcup\mathcal{S}\left(X\right)$.
\end{description}
\begin{proof}
``(i)$\Rightarrow$(ii)'': Immediate, as $\wp\left(X\right)=\ell\left(X\right)=-\ell\left(X\right)=-\wp\left(X\right)$.

``(ii)$\Rightarrow$(iii)'': Choose $x\in X$. As $x\in\wp\left(X\right)$,
by assumption, also $-x\in\wp\left(X\right)$. Hence $0\in\wp\left(X\right)$.
Then there is a minimal subset $S\subset X$ with $x\in S$ and $0\in\wp\left(S\right)$.
By definition of the simplex, $S\in\mathcal{S}\left(X\right)$.

``(iii)$\Rightarrow$(i)'': Let $x\in\ell\left(X\right)$, $x=\sum_{i=1}^{n}a_{i}x_{i}=\sum_{i=1}^{n}\left|a_{i}\right|\mathrm{sgn}\left(a_{i}\right)x_{i}$
with $x_{i}\in X$. If $\mathrm{sgn}\left(a_{i}\right)=-1$, it is
sufficient to show that $-x_{i}\in\wp\left(X\right)$ to conclude
that $x\in\wp\left(X\right)$. By assumption, each $x_{i}$ is contained
in a simplex $S_{i}$. Thus lemma \ref{lem:Sim} assures that $-x_{i}\in\ell\left(S_{i}\right)=\wp\left(S_{i}\right)\subset\wp\left(X\right)$.
\end{proof}
Further characterizations can be found in \cite{REG15}. Examples
of decomposition of positive bases into simplices can be found in
\cite[Ch 2]{SHE71}. The simplices induce a lattice structure given
by the next theorem.
\begin{thm}
\label{thm:Lat}For any positively spanning set $X$ of a linear space
$\mathbb{R}^{d}$, the set $\mathcal{L\left(X\right)}$ of subsets
positively spanning linear spaces form a boolean lattice by set inclusion.
If, moreover, $X$ is a positive basis, the lattice $\mathcal{L}\left(X\right)$
is isomorphic to $2^{\mathcal{S}\left(X\right)}$.
\end{thm}

\begin{proof}
Let $X$ be a set positively spanning $\mathbb{R}^{d}$, and, denoted
by $\mathcal{L}=\mathcal{L}\left(X\right)$, the set of subsets positively
spanning a linear subspace of $\mathbb{R}^{d}$. By convention, the
empty set spanning the null space is included. We introduce the following
order on $\mathcal{L}$.
\[
Y\sqsubset Z\Leftrightarrow\mathcal{S}\left(Y\right)\subset\mathcal{S}\left(Z\right).
\]
Our goal is to demonstrate that $\mathcal{S}:\mathcal{L}\rightarrow2^{\mathcal{S}\left(X\right)}$
is a injective homomorphism into a sublattice of $2^{\mathcal{S}\left(X\right)}$.
First, we show that for any $\mathcal{S}^{*}\subset\mathcal{S}\left(X\right)$,
$Y=\bigcup\mathcal{S}^{*}$ is a positively spanning set. By theorem
\ref{thm:Gen}, $Y$ is a positively spanning set if and only if $Y=\bigcup\mathcal{S}\left(Y\right)$.
Since $\mathcal{S}^{*}\subset\mathcal{S}\left(Y\right)$, we find
that $Y=\bigcup\mathcal{S}^{*}\subset\bigcup\mathcal{S}\left(Y\right)\subset Y$,
i.e. what needed to be shown.

Second, we show that $\mathcal{S}$ is injective. Take any $Y,Z\in\mathcal{L}$
with $\mathcal{S}\left(Y\right)=\mathcal{S}\left(Z\right)$. It is
immediate that $Y\subset X$ implies $\mathcal{S}\left(Y\right)\subset\mathcal{S}\left(X\right)$;
hence, $\mathcal{S}\left(Y\right)$ and $\mathcal{S}\left(Z\right)$
are subsets of $\mathcal{S}\left(X\right)$. From the previous result,
we obtain
\[
Y=\bigcup\mathcal{S}\left(Y\right)=\bigcup\mathcal{S}\left(Z\right)=Z.
\]

It is easy to show that $\mathcal{S}\left(Y\cap Z\right)=\mathcal{S}\left(Y\right)\cap\mathcal{S}\left(Z\right)$
and $\mathcal{S}\left(Y\right)\cup\mathcal{S}\left(Z\right)\subset\mathcal{S}\left(Y\cup Z\right)$.
From these relations and the monotonicity of $\mathcal{S}$, we find
\begin{eqnarray*}
Y\sqsubset Z & \Leftrightarrow & Y\subset Z,\\
Y\sqcap Z & = & \bigcup S\left(Y\cap Z\right),\\
Y\sqcup Z & = & Y\cup Z,
\end{eqnarray*}
where, as usual, $\sqcap$ and $\sqcup$ denote infimum and supremum,
respectively. The last equation is derived as follows:
\begin{eqnarray*}
Y\sqcup Z & = & \bigcup\left(\mathcal{S}\left(Y\right)\cup\mathcal{S}\left(Z\right)\right)\subset\bigcup\mathcal{S}\left(Y\cup Z\right)\subset Y\cup Z\\
 & = & \bigcup\mathcal{S}\left(Y\right)\cup\bigcup\mathcal{S}\left(Z\right)=\bigcup\left(\mathcal{S}\left(Y\right)\cup\mathcal{S}\left(Z\right)\right).
\end{eqnarray*}

The complement is given by
\[
Y'=\bigcup\left(S\left(X\right)\setminus S\left(Y\right)\right).
\]
From the above result, $Y'$ is well-defined as a positively spanning
set and element of $\mathcal{L}$. Altogether, we have shown that
$\mathcal{S}$ is a homomorphism of $\mathcal{L}$ ordered by $\sqsubset$
with complement $.'$ into a boolean sublattice of $2^{\mathcal{S}\left(X\right)}$.

The last statement will be proven later. Corollary \ref{cor:Iso}
will assure that for a positive basis $X$, the lattice $\mathcal{L}\left(X\right)$
is actually isomorphic to $2^{\mathcal{S}\left(X\right)}$.
\end{proof}
In the remaining part we study the relation between the simplicial
decomposition of a positive spanning set and its Gale Diagram following
the literature \cite{MAR84,SHE71}. The definition of a Gale Diagram
differs slightly in the literature up to a convention, but all definitions
are essentially interdefinable. The following definition is compatible
with all conventions. We show that the structure given by the simplices
$\mathcal{S}\left(X\right)$ and their intersections an one side,
and the Gale Diagrams on the other side, coincide in a very special
case, which has a simple characterisation in terms of the simplices.
All examples in \cite[Ch 2]{SHE71} are of this type.
\begin{defn}
Let $X\subset\mathbb{R}^{d}\setminus\left\{ 0\right\} $ be a finite
set of vectors. A dependency for $X$ is a function $v:X\rightarrow\mathbb{R}$
such that 
\[
\sum_{x\in X}v\left(x\right)\cdot x=0.
\]
Let $\mathcal{D}\left(X\right)$ denote the linear space of dependencies
of $X$, and $\mathcal{P}\left(X\right)$ the convex cone of non-negative
dependencies $v\geq0$. $X$ is said to be eqilibrated if and only
if it has a non-zero constant function as a dependency. $X$ is called
locally equilibrated if and only if every simplex $S\in\mathcal{S}\left(X\right)$
is eqilibrated. 

The Gale Diagram $\hat{x}$ for $x\in X$ is defined as follows. Select
a linear basis $v_{1},\ldots,v_{n}$ of $\mathcal{D}\left(X\right)$.
The Gale Transform of $x$ is the ray $\alpha\cdot\left(v_{1}\left(x\right),\ldots,v_{n}\left(x\right)\right)$
for $\alpha\geq0$ in $\mathbb{R}^{n}$. The Gale Diagram is the unique
point of the Gale Transform intersecting some given $n-1$-dimensional
hypersurface in $\mathbb{R}^{n}$. If the intersection is empty, then
the Gale Diagram of $x$ is $0$.
\end{defn}

Observe that for positive spanning sets, the Gale Transform of a non-zero
element can not be zero, as by theorem \ref{thm:Gen} every element
is contained in a simplex, which corresponds to a non-negative dependency.
\begin{thm}
\label{thm:Gale}Let $X\subset\mathbb{R}^{d}\setminus\left\{ 0\right\} $
be a finite set of vectors positively spanning $\mathbb{R}^{d}$.
Then the dependency space $\mathcal{D}\left(X\right)$ has a linear
basis in $\mathcal{P}\left(X\right)$. Moreover, $X$ is locally equilibrated
if and only if the following holds: Any $x,y\in X$ have the same
Gale Diagram if and only if they are contained in the same simplices
from $\mathcal{S}\left(X\right)$.
\end{thm}

\begin{proof}
For the first proposition it is sufficient to show that $\mathcal{D}\left(X\right)\subset\ell\left(\mathcal{P}\left(X\right)\right)$,
then there is a linear basis of $\mathcal{D}\left(X\right)$ contained
in $\mathcal{P}\left(X\right)$. Let $v\in\mathcal{D}\left(X\right)$.
If $v\left(x\right)\geq0$ for all $x\in X,$, then the proof is complete.
Otherwise, let $x_{k}\in X$ be with $v\left(x_{k}\right)<0$. By
theorem ..., there is an $S\in\mathcal{S}\left(X\right)$ with $x_{k}\in S=\left\{ x_{1},\ldots,x_{m}\right\} $.
Hence there are $\alpha_{1},\ldots,\alpha_{m}>0$ with $\sum_{i=1}^{m}\alpha_{i}\cdot x_{i}=0$.
Without loss of generality, we chose $\alpha_{k}=-v\left(x_{k}\right)$.
Then there is a $v_{1}\in\mathcal{P}\left(X\right)$ with $v_{1}\left(x_{i}\right)=\alpha_{i}$
and $v_{1}\left(x\right)=0$ otherwise. Then $v'=v+v_{1}\in\mathcal{D}\left(X\right)$
has $v'\left(x\right)<0$ only if $v\left(x\right)<0$, and $v'\left(x_{k}\right)=0$.
We repeat this procedure until we end with $v'\in\mathcal{P}\left(X\right)$.
We have shown that $v+\sum_{i=1}^{n}v_{i}\in\mathcal{P}\left(X\right)$,
hence $v\in\ell\left(\mathcal{P}\left(X\right)\right)$.

If $X$ is locally equilibrated, then for each simplex $S\in\mathcal{S}\left(X\right)$
the cone $\mathcal{P}\left(X\right)$ contains the characteristic
function $\chi_{S}$ of $S$. From the result above and theorem \ref{thm:Gen}
we find that the characteristic functions of simplices positively
span $\mathcal{P}\left(X\right)$. Thus, there is a linear basis of
$\mathcal{D}\left(X\right)$ consisting only of characteristic functions
$\chi_{S_{1}},\ldots,\chi_{S_{n}}$ of simplices $S_{1},\ldots,S_{n}\in\mathcal{S}\left(X\right)$,
which we choose to construct the Gale Transforms. Then, for any $S\in\mathcal{S}\left(X\right)$,
$\chi_{S}$ is a positive linear combination of the elements of the
basis. Then, clearly, whenever $x\in S\Leftrightarrow y\in S$ for
all $S\in\mathcal{S}\left(X\right)$ for $x,y\in X$, then also $\chi_{S_{i}}\left(x\right)=\chi_{S_{i}}\left(y\right)$
for $i=1,\ldots,n$, and the Gale Transform of $x$ and $y$ coincides.
Conversely, if $x$ and $y$ have the same Gale transform, then $\chi_{S_{i}}\left(x\right)=\chi_{S_{i}}\left(y\right)$
for $i=1,\ldots,n$, and $x\in S\Leftrightarrow y\in S$ for all $S\in\mathcal{S}\left(X\right)$,
which proofs the assertion.

In order to show the converse, first note that whenever $x,y\in X$
have identical Gale Transform, then there is an $\alpha>0$ such that
$v\left(x\right)=\alpha\cdot v\left(y\right)$ for each $v\in\mathcal{P}\left(X\right)$.
For each simplex $S\in\mathcal{S}\left(X\right)$ there is a dependency
$v_{S}\in\mathcal{P}\left(X\right)$ with $v_{S}\left(z\right)>0$
if and only if $z\in S$. This yields $x\in S\Leftrightarrow y\in S$
for all $S\in\mathcal{S}\left(X\right)$. Now, assume the latter equivalence
holds for some $x,y\in X$, but their Gale Transform differs. Then
there are simplices $S_{1},S_{2}\in\mathcal{S}\left(X\right)$ such
that for all $\alpha>0$
\[
\left(\begin{array}{c}
v_{S_{1}}\left(x\right)\\
v_{S_{2}}\left(x\right)
\end{array}\right)\neq\alpha\cdot\left(\begin{array}{c}
v_{S_{1}}\left(y\right)\\
v_{S_{2}}\left(y\right)
\end{array}\right).
\]
But this requires that at least one of the vectors has nonzero components.
In any case, by assumption, $x$ and $y$ are both contained in $S_{1}$
and $S_{2}$, and both vectors have nonzero components. As a consequence,
not both dependecies can be characteristic functions or positive multiples
thereof. Thus, either $S_{1}$ or $S_{2}$ is not equilibrated.
\end{proof}

\section{Positive Dependence}

We now define the skeleton of a set with subsets linearly spanning
a proper subset. This definition is consistent with the usual one
on simplicial complexes.
\begin{defn}
For any set $X$ we let
\begin{eqnarray*}
\mathrm{skel}\left(X\right) & = & \bigcup\left\{ \wp\left(A\right)\mid A\subset X,\:\ell\left(A\right)\varsubsetneq\ell\left(X\right)\right\} ,\\
\mathrm{core}\left(X\right) & = & \wp\left(X\right)\setminus\mathrm{skel}\left(X\right).
\end{eqnarray*}
\end{defn}

The skeleton is monotoneous.
\begin{lem}
If $X\subset Y$, then $\mathrm{skel}\left(X\right)\subset\mathrm{skel}\left(Y\right)$.
\end{lem}

\begin{proof}
Assume $X\subset Y$ and $z\in\mathrm{skel}\left(X\right)$. Then
there is an $A\subset X$ such that $\ell\left(A\right)\subsetneq\ell\left(X\right)$
with $z\in\wp\left(A\right)$. But then a fortiori, $A\subset Y$
and $\ell\left(A\right)\subsetneq\ell\left(Y\right)$, so $z\in\mathrm{skel}\left(Y\right)$.
\end{proof}
If $X$ is finite, then the skeleton of $X$ is a closed subset of
$X$; thus, the core is open. Its componenets are characterized by
the following lemma, which is not essential for the main results,
but for illustration purposes only.
\begin{lem}
For a finite set $X$ linearly spanning $\mathbb{R}^{d}$, each component
of the core of $X$ is contained in a set of the form $\mathrm{rint}\wp\left(B\right)$
for some linear basis $B\subset X$ of $\mathbb{R}^{d}$. If, moreover,
$X$ is positively independent, then each component is of this form.
\end{lem}

\begin{proof}
Let $X$ as above and $O$ be a component of $\mathrm{core}\left(X\right)$.
Then $O$ is open. For $x\in O$, by lemma \ref{lem:Car} there are
$x_{1},\ldots,x_{m}\in X$, $m\leq d$, with $x=\sum_{i=1}^{m}\alpha_{i}x_{i}$,
$\alpha_{i}>0$. Set $B=\left\{ x_{1},\ldots,x_{m}\right\} $. If
the dimension of $\wp\left(B\right)$ is less than $d$, then $\ell\left(B\right)\subsetneq\ell\left(X\right)$
and $x\in\wp\left(B\right)\subset\mathrm{skel}\left(X\right)$, a
contradiction. So, the dimension of $\wp\left(B\right)$ is $d$,
which (by remark \ref{rem:Bas}) requires $m=d$ and $B$ to be a
basis of $\mathbb{R}^{d}$. Hence $x\in\mathrm{rint}\wp\left(B\right)$.

For any $y\in O\setminus\mathrm{rint}\wp\left(B\right)$, there must
be a path from $x$ to $y$ contained in $O$. As the boundary of
$\wp\left(B\right)$ is contained in the skeleton of $X$, the arc
must intersect the boundary at some $z\in\mathrm{skel}\left(X\right)$,
a contradiction to $O$ being contained in the core of $X$. Thus
$O\subset\mathrm{rint}\wp\left(B\right)$.

Now, assume further that $X$ is positively independent. If $y\in\mathrm{rint}\wp\left(B\right)\cap\mathrm{skel}\left(X\right)$,
then there would be a set $A\subset X$ with $y\in\wp\left(A\right)$
and $\ell\left(A\right)\subsetneq\ell\left(X\right)=\ell\left(B\right)$.
Since $X$ is positively independent, $A\cap\wp\left(B\right)\subset B$.
Hence $y$ lies in the boundary of $\wp\left(B\right)$, a contradiction.
Thus $\mathrm{rint}\wp\left(B\right)\subset O$.
\end{proof}
The following lemma is crucial for the proof of theorem \ref{thm:Inter}.
It states two necessary and sufficient conditions for the existence
of a positive dependency within a simplex and one added element.
\begin{lem}
\label{lem:Sxy}Let $S$ be a simplex, and $y\in\ell\left(S\right)\setminus S$.
The following three propositions are equivalent:
\begin{description}
\item [{(i)}] There is an $x\in S$ with $x\in\wp\left(S\left[x\rightarrow y\right]\right)$.
\item [{(ii)}] For all $R\in\mathcal{S}\left(S\cup\left\{ y\right\} \right)$,
$\ell\left(R\right)=\ell\left(S\right)$.
\item [{(iii)}] $-y\notin\mathrm{skel}\left(S\right).$
\end{description}
\end{lem}

\begin{proof}
``(i)$\Rightarrow$(ii)'': Assume there is an $R\in\mathcal{S}\left(S\cup\left\{ y\right\} \right)$
with $\ell\left(R\right)\neq\ell\left(S\right)$. Since $y\in\ell\left(S\right)$,
this means $\ell\left(R\right)\subsetneq\ell\left(S\right)$, and
further $R\neq S$. By definition of a simplex, $R\nsubseteq S$.
Hence $y\in R$. For the same reason, $S\nsubseteq R$, thus, by lemma
\ref{lem:Sim}, $C=R\cap S$ is a basis of a proper subspace of $\ell\left(S\right)$.

Extend $C$ to a basis $B\subset S$ of $\ell\left(S\right)$, so
that $B=\left\{ x_{1},\ldots,x_{n}\right\} $, $C=\left\{ x_{1},\ldots,x_{m}\right\} $,
$m<n$, $S=B\cup\left\{ x_{0}\right\} $, and $R=C\cup\left\{ y\right\} $.
By lemma \ref{lem:Sim}, we can write
\begin{eqnarray*}
x_{0} & = & -\sum_{i=1}^{n}\delta_{i}x_{i},\\
y & = & -\sum_{i=1}^{m}\gamma_{i}x_{i},
\end{eqnarray*}
with $\delta_{i},\gamma_{i}>0$.

We have to show that no $x_{k}\in\wp\left(S\left[x_{x}\rightarrow y\right]\right)$
for $k=0,\ldots,n$. Assume the contrary. For $k>0$ we can write
with $\alpha_{i}\geq0$, $\alpha_{k}=0$, $\beta\geq0$,
\begin{eqnarray*}
x_{k} & = & \sum_{i=0}^{n}\alpha_{i}x_{i}+\beta y\\
 & = & \sum_{i=1}^{n}\left(\alpha_{i}-\beta\gamma_{i}-\alpha_{0}\delta_{i}\right)x_{i},\:\text{or}\\
x_{k}\left(1+\beta\gamma_{k}+\alpha_{0}\delta_{k}\right) & = & \sum_{\underset{i\neq k}{i=1}}^{n}\left(\alpha_{i}-\beta\gamma_{i}-\alpha_{0}\delta_{i}\right)x_{i}.
\end{eqnarray*}
As $\beta\gamma_{k}+\alpha_{0}\delta_{k}\geq0$, this means $x_{k}\in\wp\left(S\setminus\left\{ x_{k}\right\} \right)$,
contradicting the assumption that $S$ is a simplex. For the case
$k=0$, since $C\subsetneq B$, by lemma \ref{lem:Sim} one chooses
a different basis $B'\supset C$ containing $x_{0}$ and proceedes
as above. Thus $x_{k}\notin\wp\left(S\left[x_{x}\rightarrow y\right]\right)$
for $k=0,\ldots,n$.

``(ii)$\Rightarrow$(i)'': Assume that for all $R\in\mathcal{S}\left(S\cup\left\{ y\right\} \right)$
we have $\ell\left(R\right)=\ell\left(S\right)$. Since $\ell\left(S\right)=\wp\left(S\right)$
is symmetric, and $y\in\ell\left(S\right)$, by theorem \ref{thm:Gen}
there is an $R\in\mathcal{S}\left(S\cup\left\{ y\right\} \right)$
with $y\in R$. By assumption, $\ell\left(R\right)=\ell\left(S\right)$.
By definition of a simplex, we can not have $S\subset R$. So, there
is an $x\in S\setminus R$ with
\[
x\in\ell\left(S\right)=\wp\left(R\right)\subset\wp\left(S\left[x\rightarrow y\right]\right).
\]

``(ii)$\Rightarrow$(iii)'': Let $R\subset S\cup\left\{ y\right\} $
be a simplex with $\ell\left(R\right)\subsetneq\ell\left(S\right)$.
If $y\notin R$, then $R\subset S$ and, by definition, $R=S$, a
contradiction. Thus $y\in R$. By lemma \ref{lem:Sim}, $A=R\setminus\left\{ y\right\} $
is linearly independent, and $-y\in\mathrm{rint}\wp\left(A\right)$.
Since $\ell\left(A\right)\subsetneq\ell\left(S\right)$, $\wp\left(A\right)\subset\mathrm{skel}\left(S\right)$.
Hence $-y\in\mathrm{skel}\left(S\right)$.

``(iii)$\Rightarrow$(ii)'': Assume $-y\in\mathrm{skel}\left(S\right)$.
Then there is an $A\subset S$ with $-y\in\wp\left(A\right)\subset\ell\left(A\right)\subsetneq\ell\left(S\right)$.
Let $B\subset A$ be a minimal set with $-y\in\wp\left(B\right)$.
Then $-y\in\mathrm{rint}\wp\left(B\right)$, and by lemma \ref{lem:Sim},
$R=B\cup\left\{ y\right\} $ is a simplex with $\ell\left(R\right)\subset\ell\left(A\right)\subsetneq\ell\left(S\right)$.
\end{proof}
\begin{example}
Let $S=\left\{ a,b,c\right\} \subset\mathbb{R}^{2}\setminus\left\{ 0\right\} $
be a simplex. Then $\mathrm{skel}\left(S\right)=\wp\left(\left\{ a\right\} \right)\cup\wp\left(\left\{ b\right\} \right)\cup\wp\left(\left\{ c\right\} \right)$.
For $y\in-S$ we have $-y\in\mathrm{skel}\left(S\right)$ and $R=\left\{ y,-y\right\} $
is simplex in $S\cup\left\{ y\right\} $. For $y\in\ell\left(S\right)\setminus-S$
we find an $x\in S$ with $x\in\wp\left(S\left[x\rightarrow y\right]\right)$..
\end{example}

\section{Positive Bases}

We are now in a position to prove theorem \ref{thm:Inter}.
\begin{proof}
``(i)$\Rightarrow$(ii)'': We assume that condition (ii) does not
hold for some subset $Y\subset X$ and a simplex $S\in\mathcal{S}\left(X\right)$.
Let $L=\ell\left(Y\right)\cap\ell\left(S\right)$ and $M=\ell\left(Y\cap S\right)$.
Since, obviously, $M\subset L$, $M\neq L$ spells out as $M\subsetneq L$
($M$ may be null). Then we indicate that $L=M^{\bot}\oplus M$ with
some non-null linear space $M^{\bot}$ orthogonal to $M$. Pick $y\in M^{\bot}\setminus\left(S\cup-S\cup\left\{ 0\right\} \right)$,
then also $-y\in M^{\bot}\setminus\left(S\cup\left\{ 0\right\} \right)$.

Assuming both $y,-y\in\mathrm{skel}\left(S\right)$. Then, by definition,
there are sets $C,D\varsubsetneq S$ with $y\in\wp\left(C\right)\subset\ell\left(C\right)$
and $\ell\left(C\right)\subsetneq\ell\left(S\right)$, and $-y\in\wp\left(D\right)\subset\ell\left(D\right)$
and $\ell\left(D\right)\subsetneq\ell\left(S\right)$. Set $R=C\cap D$.
Since both $C$ and $D$ have at least two elements less than $S$,
$C\cup D$ is linearly independent (lemma \ref{lem:Sim}), and so
is $R$. Moreover, $\ell\left(C\right)\cap\ell\left(D\right)=\ell\left(R\right)$,
hence $y,-y\in\ell\left(R\right)$. By comparison of the coefficients
we conclude that both $y$ and $-y$ are a linear combination of elements
of $R$ with non-negative coefficients, as the coefficients of $y$
in $C$ and $-y$ in $D$ are non-negative. Hence $y,-y\in\wp\left(R\right)$.
It follows that $0\in\wp\left(R\right)\varsubsetneq\wp\left(S\right)$,
a contradiction to the fact that $S$ is a simplex (lemma \ref{lem:Sim}).

Therefore, either $y\notin\mathrm{skel}\left(S\right)$, and, by lemma
\ref{lem:Sxy}, there is an $x\in S$ with $x\in\wp\left(S\left[x\rightarrow y\right]\right)$,
or $-y\notin\mathrm{skel}\left(S\right)$, and, by the same lemma,
there is an $x\in S$ with $x\in\wp\left(S\left[x\rightarrow-y\right]\right)$,
as $\left\{ y,-y\right\} \cap S=\emptyset$ but $y,-y\in\ell\left(Y\right)\subset\wp\left(X\right)$.
In either case, $X$ is positively dependent.

``(ii)$\Rightarrow$(iii)'': Trivial.

``(iii)$\Rightarrow$(i)'': The proof is by induction over the number
of simplices contained in $X$. Let $n=\left|\mathcal{S}\left(X\right)\right|$.
For $n=1$ we find that $X$ is a single simplex, and there is nothing
left to prove. Now assume that the implication has been proven for
all proper subsets $Y$ of $X$ positively spanning a linear subspace.
Assume that (i) is false and there is an $x\in X$ with $x\in\wp\left(X\setminus\left\{ x\right\} \right)$.
Since $X$ is a positively spanning set, we can choose an $S\in\mathcal{S}\left(X\right)$
with $x\in S$. Let $Y=S'$ be the lattice complement of $S$ in $\mathcal{L}\left(X\right)$
from theorem \ref{thm:Lat}, then $Y$ is a proper positively spanning
subset, which can be written as a union of $n-1$ elements. As $X=Y\cup S$
we can write $x=y+z$ with $y\in\wp\left(Y\setminus\left\{ x\right\} \right)\subset\ell\left(Y\right)$
and $z\in\wp\left(S\setminus\left(Y\cup\left\{ x\right\} \right)\right)\subset\ell\left(S\right)$.
So, $y=x-z\in\ell\left(S\right)\cap\ell\left(Y\right)$.

Assume furthermore, that $y\in\ell\left(S\cap Y\right)$. This requires
$z=0$, and thus $x=y$. But then we have $x\in\wp\left(Y\setminus\left\{ x\right\} \right)$
and $Y$ is positively dependent, in contradiction to the induction
hypothesis which states (i) for $Y$. We conclude that $x\in\wp\left(X\setminus\left\{ x\right\} \right)$
implies that $\ell\left(Y\right)\cap\ell\left(S\right)\neq\ell\left(Y\cap S\right)$
for some positively spanning subset $Y\subset X$ and a simplex $S\in\mathcal{S}\left(X\right)$.

``(i)$\wedge$(ii)$\Rightarrow$(iv)'': As $X$ positively spans
$\mathbb{R}^{d}$, by theorem \ref{thm:Gen}, $X$ is a sum of simplices.
Let $S_{1},\ldots,S_{n}$ be a minimal set of simplices from $\mathcal{S}\left(X\right)$
with $\bigcup_{i=1}^{n}S_{i}=X$, where $X$ is a positive basis of
$\mathbb{R}^{d}$. Then each $S_{i}$ contains an element $x_{i}$
not contained in any other simplex. By lemma \ref{lem:Sim}, $A_{i}=S_{i}\setminus\left\{ x_{i}\right\} $
is linearly independent for $i=1,\ldots,n$.

We set $B=A_{1}\cup\cdots\cup A_{n}$ and show that $B$ is a linear
basis of $\mathbb{R}^{d}$. Assume, in contrast, that $B$ is linearly
dependent. Then there is an $x\in B$ with $x\in\ell\left(B\setminus\left\{ x\right\} \right)$.
Choose an $i$ with $x\in A_{i}\subset S_{i}$. As $x_{i}\notin B$,
we obtain by (ii)
\[
x\in\ell\left(B\setminus\left\{ x\right\} \right)\cap\ell\left(S_{i}\right)=\ell\left(\left(B\setminus\left\{ x\right\} \right)\cap S_{i}\right)=\ell\left(A_{i}\setminus\left\{ x\right\} \right).
\]
But this is impossible, since $A_{i}$ is linearly independent. Hence,
$B$ is linearly independent. Moreover,
\[
\ell\left(B\right)=\sum_{i=1}^{n}\ell\left(A_{i}\right)=\sum_{i=1}^{n}\ell\left(S_{i}\right)=\ell\left(X\right)=\mathbb{R}^{d},
\]
thus $B$ is a linear basis of $\mathbb{R}^{d}$. Therefore, $\left|B\right|=d$
and, further, $1\leq n\leq d$.

We show that each $S\in\mathcal{S}\left(X\right)$ has with $B$ all
but one element in common. Indeed, by (ii)
\[
\ell\left(S\right)=\ell\left(S\right)\cap\ell\left(B\right)=\ell\left(S\cap B\right).
\]
The assertion then follows from lemma \ref{lem:Sim}. Lemma \ref{lem:Sim}
also asserts that $x_{i}\in-\mathrm{rint}\wp\left(A_{i}\right)$ .

It remains to show that $A_{j}\nsubseteq A_{i}$ for $i\neq j$. Assuming
the contrary, then $x_{j}\in-\wp\left(A_{i}\right)=\wp\left(S_{i}\right)$,
and $x_{j}\notin S_{i}$. This constitutes a positive dependency within
$X$, contradicting (i). This completes the proof of (iv).

``(iv)$\Rightarrow$(iii)'': Assume that $X=B\cup\left\{ x_{1},\ldots,x_{n}\right\} $,
$1\leq n\leq d$, with a basis $B$ of $\ell\left(X\right)$, such
that each $x_{i}\in-\mathrm{rint}\wp\left(A_{i}\right)$ for a minimal
subset $A_{i}\subset B$, and $A_{i}\nsubseteq A_{j}$ for $i\neq j$.
Lemma \ref{lem:Sim} assures that $S_{i}=A_{i}\cup\left\{ x_{i}\right\} $
is a simplex. This implies that $x_{i}\neq x_{j}$ and $A_{i}\nsubseteq A_{j}$
for $i\neq j$. Since, by assumption, every simplex from $\mathcal{S}\left(X\right)$
has exactly all but one element with $B$ in common, $S_{1},\ldots,S_{n}$
are exactly the simplices in $X$.

Let $Y\subset X$ be a positively spanning set. By theorem \ref{thm:Gen},
$Y=\bigcup_{j\in J}S_{j}$ with $S_{j}\in\mathcal{S}\left(Y\right)$.
If $i\in J$, then $S_{i}\subset Y$, and $\ell\left(Y\right)\cap\ell\left(S_{j}\right)=\ell\left(Y\cap S_{j}\right)$.
Therefore, $i\notin J$. For $i\neq j$, we find that $S_{i}\cap S_{j}=A_{i}\cap A_{j}\subset B$
and $x_{i}\notin\ell\left(A_{j}\right)$ as well as $x_{j}\notin\ell\left(A_{i}\right)$.
Hence,

\begin{eqnarray*}
\ell\left(Y\right)\cap\ell\left(S_{i}\right) & = & \left(\sum_{j\in J}\ell\left(A_{j}\right)\right)\cap\ell\left(A_{i}\right)\\
 & = & \sum_{j\in J}\ell\left(A_{j}\cap A_{i}\right)\\
 & \subset & \ell\left(Y\cap S_{i}\right)\subset\ell\left(Y\right)\cap\ell\left(S_{i}\right).
\end{eqnarray*}
\end{proof}
The following result is from Reay \cite[Th 2]{REA65}, which was used
to prove the Bonnice-Klee Theorem (see also \cite[Th 10]{SHE71}).
\begin{cor}
\label{cor:Reay}Let $X$ be any positive basis of $\mathbb{R}^{d}$.
Then $X$ admits a decomposition into pairwise disjoint subsets $X=X_{1}\cup\cdots\cup X_{n}$,
$n\leq d$, such that $\left|X_{i}\right|\geq\left|X_{i+1}\right|\geq2$
for $i=1,\ldots,n-1$, and $\wp\left(X_{1}\cup\cdots\cup X_{k}\right)$
is a linear subspace of $\mathbb{R}^{d}$ of dimension $\sum_{i=1}^{k}\left|X_{i}\right|-k$.
\end{cor}

\begin{proof}
Let $B$ be the linear basis of of $\mathbb{R}^{d}$ and $A_{i}$
the sets from condition (iv). In order to remove the overlap between
the sets $A_{i}$, we set $B_{0}=\emptyset$ and $B_{i}=A_{\pi\left(i\right)}\setminus\bigcup_{j=0}^{i-1}B_{i}$
with a permutation $\pi:\left[1,\ldots,n\right]\rightarrow\left[1,\ldots,n\right]$,
for all $i=1,\ldots,,n$. $\pi$ will be chosen to ensure that the
sets $B_{i}$ are ordered in non-increasing cardinality. Hence $B_{i}\cap B_{j}=\emptyset$
for $i\neq j$ and $\left|B_{i}\right|\geq\left|B_{i+1}\right|$.
Since $\bigcup_{i=1}^{n}A_{i}=B$, we find $B=B_{1}\cup\cdots\cup B_{n}$

Set $X_{i}=B_{i}\cup\left\{ x_{i}\right\} $. As $\left|X_{i}\right|=\left|B_{i}\right|+1$,
we obtain the theorem \ref{thm:Gen}
\[
\wp\left(X_{1}\cup\cdots\cup X_{k}\right)=\wp\left(S_{1}\cup\cdots\cup S_{k}\right)=\ell\left(B_{1}\cup\cdots\cup B_{k}\right),
\]
which implies the assertion.
\end{proof}
The special case when the simplices span disjoint subspaces is covered
by \cite[Th 8]{SHE71}.

Although this statement could be used to construct the basis $B$
in (iv), our result proves more, by, namely, showing that each simplex
has all but one element with $B$ in common (see counterexample \ref{exa:X9}).
This insight immediately yields the missing part of the proof of theorem
\ref{thm:Lat}.
\begin{cor}
\label{cor:Iso}Let $X$ be a positive basis of $\mathbb{R}^{d}$.
Then $\mathcal{S}\left(X\right)$ is a minimal set of simplices spanning
$\mathbb{R}^{d}$. Furthermore, the lattice $\mathcal{L}\left(X\right)$
is isomorphic to $2^{\mathcal{S}\left(X\right)}$.
\end{cor}

\begin{example}
\label{exa:X9}The following set $X=\left\{ x_{1},\ldots,x_{9}\right\} $
positively spans $\mathbb{R}^{6}$ and satisfies the condition of
corollary \ref{cor:Reay}, with simplices $X_{1}=\left\{ x_{1},x_{2},x_{3}\right\} $,
$X_{2}=\left\{ x_{4},x_{5},x_{6}\right\} $, and $X_{3}=\left\{ x_{7},x_{8},x_{9}\right\} $,
but is positively dependent, as $x_{6}=x_{1}+x_{2}+x_{8}+x_{9}$.
We observe that $X_{1}\cup X_{2}$ spans a $4$-dimensional subset,
and $S=\left\{ x_{3},x_{6},x_{7}\right\} $ is another simplex, violating
the last condition in theorem \ref{thm:Inter} (iv) for any basis
$B\subset X$ (as not all four simplices can have one element not
in $B$). Hence, this last condition is unable to be removed.
\[
x_{1}=\left(\begin{array}{c}
1\\
0\\
0\\
0\\
0\\
0
\end{array}\right),x_{2}=\left(\begin{array}{c}
0\\
1\\
0\\
0\\
0\\
0
\end{array}\right),x_{3}=\left(\begin{array}{c}
-1\\
-1\\
0\\
0\\
0\\
0
\end{array}\right),x_{4}=\left(\begin{array}{c}
0\\
0\\
-1\\
0\\
-1\\
0
\end{array}\right),x_{5}=\left(\begin{array}{c}
0\\
0\\
0\\
-1\\
-1\\
0
\end{array}\right),
\]
\[
x_{6}=\left(\begin{array}{c}
0\\
0\\
1\\
1\\
2\\
0
\end{array}\right),x_{7}=\left(\begin{array}{c}
1\\
1\\
-1\\
-1\\
-2\\
0
\end{array}\right),x_{8}=\left(\begin{array}{c}
-1\\
-1\\
1\\
1\\
0\\
1
\end{array}\right),x_{9}=\left(\begin{array}{c}
0\\
0\\
0\\
0\\
2\\
-1
\end{array}\right).
\]
\end{example}

\section{Conical Decompositions}

This chapter contains the proof of the main theorem \ref{thm:Main}.
\begin{lem}
\label{lem:MDim}Let $X$ be a set positively spanning $\mathbb{R}^{d}$.
Then each $A\in\mathcal{M}\left(X\right)$ linearly spans $\mathbb{R}^{d}$.
\end{lem}

\begin{proof}
Let $A\in\mathcal{M}\left(X\right)$ and set $M=\ell\left(A\right)\subset\mathbb{R}^{d}=\ell\left(X\right)$.
If $M\subsetneq\mathbb{R}^{d}$, then there is a hyperplane $H$ of
$\mathbb{R}^{d}$ containing $M$. Let $z$ be a normal vector of
$H$. Since $X$ positively generates $L$, there must be an $x\in X$
lying in the open half-space in direction $z$. Clearly, $x$ is not
in $A$. Since $A$ is negatively independent, lemma \ref{lem:NegInd}
warrants the existentence of a vector $w$ with $w\cdot y>0$ for
all $y\in A$. By construction, $y\cdot z=0$ for all $y\in A$, and
$x\cdot z>0$. Then there is an $\varepsilon>0$ such that $\varepsilon\cdot w\cdot x<z\cdot x$.
For $u=z+\varepsilon\cdot w$ we have $u\cdot x=z\cdot x+\varepsilon\cdot w\cdot x>0$,
and for $y\in A$ we obtain $u\cdot y=\varepsilon\cdot w\cdot y>0$.
Thus, $A\cup\left\{ x\right\} $ is negatively independent, a contradiction
to the assumption that $A$ is a maximally negatively independent
subset of $X$. Hence $M=\mathbb{R}^{d}$, or $\ell\left(A\right)=\ell\left(X\right)$.
\end{proof}
The following lemma constitutes the crucial combinatorical insight
for our main theorem.
\begin{lem}
\label{lem:MaxInd}Let $X$ be a positively spanning set. If $A$
contains all but one element of each simplex $S\in\mathcal{S}\left(X\right)$,
then $A\in\mathcal{M}\left(X\right)$. If $X$ is positively independent,
then the converse holds.
\end{lem}

\begin{proof}
Assume first, $\left|S\cap A\right|=\left|S\right|-1$ for all simplices
$S\in\mathcal{S}\left(X\right)$. If $A$ is not negatively independent,
then by lemma \ref{lem:NegInd} there would be a simplex $S\subset A\subset X$,
contradicting the assumption. Now let $B$ be a negatively independent
set with $A\subset B\subset X$. For $x\in B\setminus A$, since $X=\bigcup\mathcal{S}\left(X\right)$,
there is a simplex $S\in\mathcal{S}\left(X\right)$ with $x\in S$
and $\left|S\cap A\right|=\left|S\right|-1$, hence $\left(S\cap A\right)\cup\left\{ x\right\} =S$.
But this implies $S\subset B$, in contradiction to the assumption
that $B$ is negatively independent. We conclude $A=B$. So, $A\in\mathcal{M}\left(X\right)$.

Conversely, assume that $X$ is positively independent, let $A\in\mathcal{M}\left(X\right)$
be a maximally negatively independent subset of $X$ and consider
any simplex $S\in\mathcal{S}\left(X\right)$ . So, $S$ can not be
a subset of $A$. Select $x\in S\setminus A$. By lemma \ref{lem:Sim}
$B=S\setminus\left\{ x\right\} $ is a basis spanning $\ell\left(S\right)$.
By lemma \ref{lem:MDim}, $A$ linearly spans $\ell\left(X\right)$.
Theorem \ref{thm:Inter} yields
\[
\ell\left(B\right)=\ell\left(S\right)\cap\ell\left(X\right)=\ell\left(S\right)\cap\ell\left(A\right)=\ell\left(S\cap A\right)=\ell\left(B\cap A\right),
\]
so $B\subset A$, as $B$ is linearly independent. This gives us $\left|S\cap A\right|=\left|S\right|-1$.
\end{proof}
The converse proposition does not hold if $X$ is positively dependent,
as the following example shows.
\begin{example}
Let $S$ be a simplex of any dimension greater one and set $X=S\cup-S$.
Then for any $x\in S$, $S\left[x\rightarrow-x\right]$ is in $\mathcal{M}\left(X\right)$,
and its intersection with the only other simplex $-S$ consists only
of the element $-x$.
\end{example}

\begin{proof}
Let $S$ and $X$ be as above. Then also $-S$ is a simplex. For $x\in S$,
by lemma \ref{lem:Sim}, $-x\in\wp\left(S\right)$, and $S\setminus\left\{ x\right\} $
is negatively independent, and so is $A=S\left[x\rightarrow-x\right]$.
For any $y\in X\setminus A$ either $y=x$ or $y\in-S\setminus\left\{ -x\right\} $.
In any case, $-y\in A$, thus $A\cup\left\{ y\right\} $ is negatively
dependent. Hence $A\in\mathcal{M}\left(X\right)$, and $A\cap-S=\left\{ -x\right\} $.
\end{proof}
Next is a simple inequality.
\begin{lem}
\label{lem:Prod}For natural numbers $k_{1},\ldots,,k_{n}\geq1$ with
$\sum_{i=1}^{n}k_{i}=d$ we have
\[
\prod_{i=1}^{n}k_{i}\leq2^{d-n}.
\]
The equality is strict whenever one $k_{i}$ is greater than one.
\end{lem}

\begin{proof}
The case $n=1$ follows from $k\leq2^{k-1}$ for $k\geq1$ with equality
for $k=1$ only. Assume now that the above formula has been proven
for $n-1$. Then
\[
\prod_{i=1}^{n}k_{i}\leq2^{d-k_{n}-\left(n-1\right)}\cdot k_{n}\leq2^{d-n}.
\]
Again, the inequality is strict, if $k_{n}>1$.
\end{proof}
Now to the proof of the main theorem \ref{thm:Main}.
\begin{proof}
Let $X$ be a positive basis. By theorem \ref{thm:Inter} (iv), we
can write $X=B\cup\left\{ x_{1},\ldots,x_{n}\right\} $, $1\leq n\leq d$,
with a basis $B$ of $\ell\left(X\right)=\mathbb{R}^{d}$, such that
each $x_{i}\in-\mathrm{rint}\wp\left(A_{i}\right)$ for $A_{i}\subset B$,
and $A_{i}\neq A_{j}$ for $i\neq j$. Therefore, $\left|B\right|=d$.
Moreover, the theorem assures that each simplex $S\in\mathcal{S}\left(X\right)$
has all but one element in common with $B$. Hence $S_{i}=A_{i}\cup\left\{ x_{i}\right\} $
, $i=1,\ldots,n$, are exactly the simplices in $\mathcal{S}\left(X\right)$.

As in the proof of theorem \ref{thm:Inter} (iv), in order to remove
the overlap between those sets, we set $B_{0}=\emptyset$ and $B_{i}=A_{i}\setminus\bigcup_{j=0}^{i-1}B_{i}$
. Hence $B_{i}\cap B_{j}=\emptyset$ for $i\neq j$.

First, we show that each $B_{k}$ is non-empty for $k\geq1$. Assume
$B_{k}=\emptyset$. Then $A_{k}\subset\bigcup_{i=1}^{k-1}A_{i}$,
and
\[
x_{k}\in\ell\left(A_{k}\right)\subset\ell\left(\bigcup_{i=1}^{k-1}A_{i}\right)=\wp\left(\bigcup_{i=1}^{k-1}S_{i}\right).
\]
The last equation holds because for $x\in A_{i}$, $-x\in\wp\left(S_{i}\right)$.
As $x_{k}\notin\bigcup_{i=1}^{k-1}S_{i}$, which indicates a positive
dependence within $X$, a contradiction. Hence $\left|B_{i}\right|\geq1$
for $i=1,\ldots,n$.

Since $X=B\cup\left\{ x_{1},\ldots,x_{n}\right\} $ with $B\cap\left\{ x_{1},\ldots,x_{n}\right\} =\emptyset$
and $1\leq n\leq d$, we find well-known inequalities \cite[Th. 3.8 and 6.7]{DAV54}
\[
d+1\leq\left|X\right|=d+n\leq2d.
\]
As $B$ is a basis of $\mathbb{R}^{d}$, we must have $\left|B\right|=d$.
For $\left|X\right|=2d$, we find $n=d$, which requires $\left|B_{i}\right|=1$
for all $i=1,\ldots,n$, such that $X$ is a union of $d$ 1-simplices
forming a cross. The lower bound $\left|X\right|=d+1$ yields $n=1$,
and $X=S_{1}$ is a single simplex.

We now procede to give the upper bound of the number of maximally
negatively independent subsets of $X$. Let $X_{i}=B_{i}\cup\left\{ x_{i}\right\} $,
$i=1,\ldots,n$. Since $x_{i}\notin B_{j}$ we find that the $X_{i}$
are a pairwise disjoint decomposition of $X$. Let $A\in\mathcal{M}\left(X\right)$
. By lemma \ref{lem:MaxInd}, each set has with each simplex $S_{i}$
all but one element in common. We show that $A$ can be constructed
by picking one element from each $X_{i}$, which is to be excluded
from $A$. At the same time, we establish the upper bound for a number
of those constructions. Since $X_{i}\subset S_{i}$, $A$ contains
either all or all but one elements of $X_{i}$.

As $X_{1}=S_{1}$, we can pick any $y_{1}\in X_{1}$, and there are
exactly $\left|X_{1}\right|$ choices. Thus, $A$ must contain $X_{1}$
except for one $y_{1}\in X_{1}$. Upon constructing $A$, we start
with $A=X_{1}\setminus\left\{ y_{1}\right\} $. Assume now we have
picked $y_{1},\ldots,y_{j}$ from $X_{1},\ldots,X_{k-1}$, $j\leq k-1<n$,
and established that $A$ contains all of $S_{i}$, $1\leq i<k$,
except for one element from the list $y_{1},\ldots,y_{j}$. Now, going
from $k-1$ to $k$, three cases have to be distinguished.

(i) If $S_{k}$ contains more than one element from the list $y_{1},\ldots,y_{j}$,
then the list is incompatible with the goal of constructing an element
of $\mathcal{M}\left(X\right)$. There are zero choices. Since we
are looking for the upper bound of possible choices, there is no problem
with this case (which can not occur if $A$ is already given as a
maximal negatively independent set). (ii) If $S_{k}$ contains exactly
one element $y_{i}$ from the list, then $A$ does not have this element
or should not have in common with $S_{k}$, and $A$ contains all
the other elements from $S_{k}$ (upon constructing $A$, we are adding
all elements of $X_{k}\setminus\left\{ y_{i}\right\} $ to it). (iii)
If $X_{k}$ does not contain an element from the list, then any element
$y_{j+1}\in X_{k}$ can be chosen and added to the list, and there
are $\left|X_{k}\right|$ choices. In the latter two cases, since
$S_{k}\subset X_{1}\cup\cdots\cup X_{k}$, we are assured that $S_{k}$
has all but one element in common with $A$. At each step, there are
at most $\left|X_{k}\right|$ choices, which remains true also in
case (i).

As $d=\left|B\right|=\sum_{i=1}^{n}\left|B_{i}\right|$ and $\left|B_{i}\right|\geq1$,
we find with lemma \ref{lem:Prod}
\[
\left|\mathcal{M}\left(X\right)\right|\leq\prod_{i=1}^{n}\left|X_{i}\right|=\prod_{i=1}^{n}\left(\left|B_{i}\right|+1\right)\leq2^{n}\prod_{i=1}^{n}\left|B_{i}\right|\leq2^{n}\cdot2^{d-n}=2^{d}.
\]
By the same lemma, the inequality is strict save for $\left|B_{i}\right|=1$
for $i=1,\ldots,n$. This requires all $S_{i}$ to be 1-simplices.
But this also yields $n=d$. Hence, equality holds in the above relation
if and only if $X$ is a cross. 

If $X$ consists of one simplex, then $\left|\mathcal{M}\left(X\right)\right|=d+1$.
Otherwise, if $n>1$ it is easy to see that $\left|\mathcal{M}\left(X\right)\right|>d+1$,
hence the lower bound of the last inequality, which holds with equality
if and only if $X$ is a simplex. 
\end{proof}

\section{Consequences of the main theorem}

The last inequality of the main theorem \ref{thm:Main} can be extended
to arbitrary positively spanning sets.
\begin{thm}
\label{thm:MainExt}Let $X$ be a set positively spanning $\mathbb{R}^{d}$.
Then the elements of $X$ can be subdivided into at most $2^{d}$
negatively independent sets. The boundary is met exactly in case that
$X$ is a sum of 1-simplices.
\end{thm}

\begin{proof}
Assume $X$ is positively spanning $\mathbb{R}^{d}$. Pick a positive
basis $Y\subset X$ \cite[Th. 4.3]{REG15}. Consider $x\in X$. Since
$Y$ positively spans $L$, there is a negatively independent set
$A\subset Y$ with $x\in\wp\left(A\right)$. But $A$ can be expanded
to a maximally negatively independent set $B\in\mathcal{M}\left(Y\right)$.
So, every $x\in X$ lies in the positive span of some element of $\mathcal{M}\left(Y\right)$,
and there are at most $2^{d}$ of them.

Now assume that $X$ can not be subdivided into less then $2^{d}$
negatively independent sets. Then $\mathcal{S}\left(X\right)$ does
not contain any simplex, which is not a 1-simplex. For if there is
a simplex $S_{0}\in\mathcal{S}\left(X\right)$, which is not a 1-simplex,
we can construct this basis by extending $S_{0}$ to a minimal set
that positively spans $\mathbb{R}^{d}$, and has $\left|\mathcal{M}\left(Y\right)\right|<2^{d}$,
a contradiction to the assumption. Hence $X$ can be written as a
sum of 1-simplices, of which $d$ must form a cross.
\end{proof}
The next lemma prepares for the last major result.
\begin{lem}
\label{lem:MRed}Let $X,Y$ be sets positively spanning $\mathbb{R}^{d}$
with $Y\subset X$. Then $A\in\mathcal{M}\left(Y\right)$ if and only
if there is a $B\in\mathcal{M}\left(X\right)$ with $A=B\cap Y$.
Moreover, for $A,B\in\mathcal{M}\left(X\right)$, $A\cap Y=B\cap Y$
implies $\ell\left(A\cap B\right)=\mathbb{R}^{d}$.
\end{lem}

\begin{proof}
Let $B\in\mathcal{M}\left(X\right)$ and set $A=B\cap Y$. By lemma
\ref{lem:MaxInd}, for every $S\in\mathcal{S}\left(X\right)$, $S\cap B$
contains all but one element of $S$. For $S\in\mathcal{S}\left(Y\right)$,
by lemma , $S\in\mathcal{S}\left(X\right)$, and $S\cap A=S\cap B\cap Y=S\cap B$
contain all but one element from $S$. Hence $A\in\mathcal{M}\left(Y\right)$.

Conversely, let $A\in\mathcal{M}\left(Y\right)$. Since $A$ is a
negatively independent subset of $X$, there is a maximal extension
$B\in\mathcal{M}\left(X\right)$ with $A\subset B$. Clearly, $A\subset B\cap Y$.
If there existed an $x\in\left(B\cap Y\right)\setminus A$, then $A\cup\left\{ x\right\} $
would be negatively independent, a contradiction to $A$ being a maximally
independent subset of $Y$. Hence $A=B\cap Y$.

To show the second statement, assume $A,B\in\mathcal{M}\left(X\right)$,
$A\cap Y=B\cap Y$, and $\ell\left(X\right)=\ell\left(Y\right)=\mathbb{R}^{d}$.
Then $\ell\left(S\cap A\right)=\ell\left(S\right)$ for every $S\in\mathcal{S}\left(Y\right)\subset\mathcal{S}\left(X\right)$,
and further
\[
\ell\left(A\cap Y\right)=\ell\left(\bigcup_{S\in\mathcal{S}\left(Y\right)}A\cap S\right)=\ell\left(\bigcup_{S\in\mathcal{S}\left(Y\right)}S\right)=\ell\left(Y\right).
\]
So,
\[
\ell\left(A\cap B\right)\subset\ell\left(Y\right)=\ell\left(A\cap Y\right)=\ell\left(A\cap B\cap Y\right)\subset\ell\left(A\cap B\right),
\]
which assures $\ell\left(A\cap B\right)=\ell\left(Y\right)$.
\end{proof}
In general, even in two dimensions, $\mathcal{M}\left(X\right)$ can
be arbitrarily large.
\begin{example}
For any $n\geq1$ there is a set $X_{n}\subset\mathbb{R}^{2}$ with
$\left|\mathcal{M}\left(X_{n}\right)\right|=\left|X_{n}\right|=2n$.
\end{example}

To see this, divide the unit circle into $2n$ segments by taking
$n$ antipodal pairs of points with equal distances to the neighbour.
Then the maximal negatively independent sets are exactly those consisting
of $n$consecutive points, and there are exactly $2n$ of them. Thus,
the intersection condition in theorem \ref{thm:Max} is necessary
for the inequality.

We are now in a position to prove theorem \ref{thm:Max}.
\begin{proof}
Let Let $\mathcal{A}\subset\mathcal{M}\left(X\right)$, where $X$
positively spans $\mathbb{R}^{d}$, such that for all $A,B\in\mathcal{A}$,
$A\neq B$, $\wp\left(A\cap B\right)$ does not have full dimension.
Choose a positive basis $Y\subset X$. By lemma \ref{lem:MRed}, for
each $B\in\mathcal{A}$, $B\cap Y\in\mathcal{M}\left(Y\right)$. Moreover,
for $A,B\in\mathcal{A}$ we do not have $A\cap Y=B\cap Y$, because
this would imply $\ell\left(A\cap B\right)=\mathbb{R}^{d}$, contradicting
the assumption that $\wp\left(A\cap B\right)$ does not have full
dimension. Hence $\left|\mathcal{A}\right|\leq\left|\mathcal{M}\left(Y\right)\right|\leq2^{d}$.

Now consider the case $\left|\mathcal{A}\right|=2^{d}$. If $Y=\bigcup\mathcal{A}$
does not span the whole space, then by the separating hyperplane theorem,
$Y$ is contained in a half-space with the boundary containing the
origin. Then there must be another element of $\mathcal{M}\left(X\right)$
contained in the opposite half space, which could be added to $\mathcal{A}$,
contradicting the last inequality. Moreover, we must have $\bigcup\mathcal{A}=X$,
otherwise any $x\in X\setminus Y$ must be contained in the positive
span of one set in $\mathcal{A}$, contradicting its maximality. If
$X$ contained a positive basis, which is not a cross, then by theorem
\ref{thm:MainExt}, $\left|\mathcal{A}\right|<2^{d}$. Thus, $X$
does not contain any simplices apart from 1-simplices. So, $X$ is
a union of at least $d$ 1-simplices, of which some $d$ of them form
a cross.
\end{proof}

\end{document}